\documentclass{svmult}
\usepackage{makeidx}         
\usepackage{graphicx}        
\usepackage{multicol}        
\usepackage[bottom]{footmisc}

\RequirePackage{amssymb}
\RequirePackage{amsmath}

\makeindex             

\newtheorem{thm}{Theorem}[section]
\newtheorem{lem}[thm]{Lemma}

\newtheorem{conj}[thm]{Conjecture}

\begin{document}

\baselineskip14pt

\par\vspace{4mm}

\title*{On the Chudnovsky-Seymour-Sullivan Conjecture on Cycles in Triangle-free Digraphs}
\titlerunning {Cycles in Triangle-free Digraphs}
\author{Kevin Chen\inst{1}, Sean Karson\inst{2}, Dan Liu\inst{3}, Jian Shen\inst{4}\thanks{Corresponding author. Shen was partially supported by NSF (CNS 0835834) and Texas Higher Education Coordinating Board (ARP 003615-0039-2007).}}
\authorrunning{Chen, Karson, Liu, and Shen}
\institute{Clements High School, Sugar Land, TX 77479   \and
Trinity Preparatory School, Winter Park, FL 32792 \and
Liberal Arts and Science Academy, Austin, TX 78724
\and Department of
Mathematics, Texas State
University, San Marcos, TX 78666\\
\texttt{js48@txstate.edu}}
\maketitle

\begin{abstract}
For a simple digraph $G$ without directed triangles or digons, let $\beta(G)$ be the size of the smallest subset $X \subseteq E(G)$ such that $G\setminus X$ has
no directed cycles, and let $\gamma(G)$ be the number of unordered pairs of nonadjacent vertices in $G$. In 2008, Chudnovsky, Seymour, and Sullivan
showed that $\beta (G) \le \gamma(G)$, and conjectured that $\beta (G) \le \gamma(G)/2$.
Recently, Dunkum, Hamburger, and P\'or proved that $\beta (G) \le 0.88 \gamma(G)$. In this note,
we prove that  $\beta (G) \le 0.8616 \gamma(G)$.
\end{abstract}

\section{Introduction}

We will follow the notation from \cite{ChudnovskySS, DunkumHP}. All digraphs $G=(V,E)$ considered in this note are finite and simple. A digraph $G$ is called $3$-{\em free} if $G$ has no directed cycle of length at most three. A digraph is {\em acyclic} if it has no directed cycles. For a digraph $G$, let $\beta(G) $ denote the minimum cardinality of a set $X \subset E(G)$ such that $G\setminus X$ is acyclic, and let $\gamma(G)$ be the number of missing edges of $G$ (that is, the number of unordered pairs of nonadjacent vertices.) In 2008, Chudnovsky, Seymour, and Sullivan \cite{ChudnovskySS} made the following conjecture.

\begin{conj} [Chudnovsky, Seymour, and Sullivan] \label{ChudnovskySS}
If $G$ is a $3$-free digraph, then $$\beta (G) \le \frac  1 2 \gamma(G).$$
\end{conj}

In support of the above conjecture,
Chudnovsky, Seymour, and Sullivan \cite{ChudnovskySS} showed that $\beta (G) \le \gamma(G)$. Recently,
Dunkum, Hamburger, and P\'or \cite{DunkumHP} improved the result to $\beta (G) \le 0.88\gamma(G)$.
Conjecture \ref{ChudnovskySS} is closely related to the following special case of a conjecture by Caccetta and H${\rm\ddot{a}}$ggkvist \cite{CH}.

\begin{conj} [Caccetta and H${\rm\ddot{a}}$ggkvist] \label{CH}
Any digraph on $n$ vertices with minimum outdegree at least $n/3$ contains a
directed triangle.
\end{conj}

Chudnovsky, Seymour, and Sullivan \cite{ChudnovskySS} commented that proving Conjecture \ref{ChudnovskySS}
may provide some useful information towards proving Conjecture \ref{CH}. To see this, their partial result
($\beta (G) \le \gamma(G)$) on Conjecture \ref{ChudnovskySS} has been applied by Hamburger, Haxell, and Kostochka \cite{HamburgerHK}
to improve a result of Shen \cite{Shen} on Conjecture \ref{CH}. Recently, the same partial result was also applied by Hladk\'y, Kr\'al', and Norine \cite{HladkyKN} who used the theory of flag algebras to prove the currently best result in this direction, namely, any digraph on $n$ vertices with minimum outdegree at least $0.3465n$ contains a directed triangle.
In this note, we prove that $\beta (G) \le 0.8616\gamma(G)$.

\section{Proof of the Main Result}
In this section, we follow the ideas in \cite{ChudnovskySS, DunkumHP} for partitioning the vertex set of a digraph. For each vertex $v$ in $G$, let
$A(v)$ and $B(v)$ be the set of out-neighbors  and the set of in-neighbors of $G$, respectively.
Then there are no edges from $A(v)$ to $B(v)$; or else, $G$ would contain a directed triangle.
Let $g(v)$ be the number of missing edges between $A(v)$ and $B(v)$.
Denote $C(v):=V-A(v)-B(v)-\{v\}$. Dunkum, Hamburger, and P\'or \cite{DunkumHP} partitioned $V$ into $V_1,$ $V_2,$ $\{v \}$ such that $V_1= B(v) \cup C_{B(v)}$ and $V_2= A(v) \cup C_{A(v)}$, where $C_{A(v)} \cup C_{B(v)}$ forms a certain partition of $C(v)$. Given such a partition $V_1\cup V_2 \cup \{v\}$ of $V$,
let $G[V_1]$ and $G[V_2]$ be the subgraphs induced by $V_1$ and by $V_2$, respectively.
The edges which are missing outside of $G[V_1]$ and $G[V_2]$ are denoted as {\em missing edges}. Note that removing the set of edges from $V_2$ to $V_1$ destroys all directed cycles outside of $G[V_1]$ and $G[V_2]$. Thus the edges
from $V_2$ to $V_1$ are called {\em decycling edges}. An easy induction argument \cite{ChudnovskySS, DunkumHP} shows that, for any real $\mu$ with $ 0 \le \mu \le 1$, if
the number of missing edges is at least $(1+\mu)$ times the number of decycling edges, then $\gamma (G) \ge (1+\mu) \beta(G).$ (See the proof of Theorem~\ref{Main}.)
The following two lemmas are due to Dunkum, Hamburger, and P\'or \cite{DunkumHP}.

\begin{lem}  [\cite{DunkumHP}] \label{Lemma3.4}
If $ \sum_{v \in V(G)}|C(v)|+ \frac 12 \sum_{v \in V(G)} {{|C(v)|} \choose 2} + \frac {1-\mu } 4 \sum_{
v \in V(G)} t(v) \ge \mu  \sum_{v \in V(G)} g(v)$,
then for some vertex $v$ there exists a partition $V_1, V_2, \{v \}$ where
the number of missing edges is at least $(1+\mu)$ times the number of decycling edges.
\end{lem}

\begin{lem} [\cite{DunkumHP}] \label{Lemma3.1} If $$g(v) \ge |C(v)|^2 (1+\mu) \left ( \frac {1+\mu + \sqrt{ (1+\mu)^2 + 1+\mu}} 2 + \frac 1 4 \right )$$
for a
vertex $v$, then there exists a partition $V_1$, $V_2$, $\{v \}$ where the number of
missing edges is at least $(1+\mu)$ times the number of decycling edges.
\end{lem}

Let $e(v)$ be the number edges from $C_{A(v)}$ to $C_{B(v)}$. The next lemma is a modification of Lemma~\ref{Lemma3.1}.
The proof of Lemma~\ref{ModifiedLemma3.1} is quite similar
to the proof of Lemma~\ref{Lemma3.1} in \cite{DunkumHP}. To make the note self-contained, we include a proof.


\begin{lem} \label{ModifiedLemma3.1}  If $$g(v) \ge |C(v)|^2 (1+\mu) \left (
\frac {1+\mu + \sqrt{ (1+\mu)^2 + \frac { 4(1+\mu) e(v)} {|C(v)|^2}}} 2 + \frac {e(v)} {|C(v)|^2} \right )$$ for a
vertex $v$, then there exists a partition $V_1$, $V_2$, $\{v \}$ where the number of
missing edges is at least $(1+\mu)$ times the number of decycling edges.
\end{lem}

\begin{proof} Following the ideas in \cite{DunkumHP}, we partition
the vertex set of $G$ into $V_1$, $V_2$, $\{v\}$ as follows. First let $B(v) \subseteq V_1$
and $A(v) \subseteq V_2$. Second, for any $u \in C(v)$, let $k_v(u)$ (resp.~$l_v(u)$) be the number of vertices $w \in A(v)$ (resp.~$w \in B(v)$) with $wu \in E(G)$
(resp.~$uw \in E(G)$), and further let $u \in V_1$ if $l_v(u) > k_v (u)$ and let $u \in V_2$ otherwise.
Denote the two subsets of $C(v)$ by $C_{A(v)}$ and $C_{B(v)}$; that is,
$C_{A(v)} = C(v) \cap V_2$ and $C_{B(v)} = C(v) \cap V_1$. Denote $m_v(u): =\min \{ k_v(u), l_v(u)\}$
and $M:= \sum_{v \in C(v)} m_v(u)$.

\vspace{.3cm}

For each $u \in C(v)$, there are $k_v(u)$ and $l_v(u)$ edges from $A(v)$ to $v$ and from $v$ to $B(v)$, respectively. Denote the two sets by $K_v(u) \subseteq A(v)$ and $L_v (u) \subseteq B(v)$. Any edge from $L_v(u)$ to $K_v(u)$ would form a directed triangle together with $v$. Thus these $k_v(u)l_v(u)$ edges between $K_v(u)$ and $L_v(u)$ are missing. Each missing edge between $A(v)$ and $B(v)$ can be counted with multiplicity at most $|C(v)|$ in the sum $\sum_{u \in C(v)} k_v(u)l_v(u)$. This yields a lower bound for the number of missing edges $g(v)$ between  $A(v)$ and $B(v)$:
\begin{equation} \label{g(v)}
g(v) \ge \frac 1 {|C(v)|} \sum_{u \in C(v)} k_v(u)l_v(u) \ge \frac 1 {|C(v)|} \sum_{u \in C(v)} m_v^2(u)
\ge \left ( \frac {\sum_{u \in C(v)} m_v(u)}{|C(v)|} \right )^2 = \frac {M^2}{|C(v)|^2}.
\end{equation}
To count the number of decycling edges, we see that there are three types of decycling edges: edges from $A(v)$ to $C_{B(v)}$, edges from $C_{A(v)}$ to $B(v)$, and edges from $C_{A(v)}$ to $C_{B(v)}$. The number of decycling edges of the first two types is $M$. Recall that $e(v)$ is the number of edges from $C_{A(v)}$ to $C_{B(v)}$. So the total number of decycling edges is $M + e(v)$. If
$$M \le \frac {1+\mu + \sqrt{ (1+\mu)^2 + \frac { 4(1+\mu) e(v)} {|C(v)|^2}}} 2 |C(v)|^2,$$
then
$$g(v) \ge  |C(v)|^2 (1+\mu) \left (
\frac {1+\mu + \sqrt{ (1+\mu)^2 + \frac { 4(1+\mu) e(v)} {|C(v)|^2}}} 2 + \frac {e(v)} {|C(v)|^2} \right ) \ge (1+\mu) (M + e(v))$$
and we are done.
Now we may suppose
$$M \ge \frac {1+\mu + \sqrt{ (1+\mu)^2 + \frac { 4(1+\mu) e(v)} {|C(v)|^2}}} 2 |C(v)|^2,$$
which implies
\begin{equation} \label{quadratic}
\frac {M^2}{|C(v)|^4} - \frac {(1+\mu) M}{|C(v)|^2} - \frac {(1+\mu)e(v)}{|C(v)|^2} \ge 0.
\end{equation}
By (\ref{g(v)}) and (\ref{quadratic}),
$$g(v) \ge \frac {M^2}{|C(v)|^2} \ge (1+\mu) (M + e(v)),$$
from which Lemma~\ref{ModifiedLemma3.1} follows.
\hspace*{\fill}$\Box$
\end{proof}

\begin{thm} \label{condition} Let $\mu$ be a positive real satisfying the four inequalities:
\begin{itemize}
\item[(I)] \hspace{.4cm} $4\mu^2 + 5 \mu -1 \le 0$,
\item[(II)] \hspace{.4cm} $24\mu^4 + 49 \mu^3 + 8 \mu^2 -19 \mu +2 \le 0,$
\item[(III)] \hspace{.4cm} $8\mu^3+20\mu^2+13\mu-5 \le 0,$ and
\item[(IV)] \hspace{.4cm} $32\mu^4 - 8 \mu^3 -159 \mu^2 -130 \mu +25 \ge 0.$
\end{itemize}
Then there exists
a vertex $v$ and a partition $V_1$, $V_2$, $\{v \}$ where the number of
missing edges is at least $(1+\mu)$ times the number of decycling edges.
\end{thm}

\begin{proof} By Lemmas~\ref{Lemma3.4}, we may assume that
$$
\sum_{v \in V(G)}|C(v)|+ \frac 12 \sum_{v \in V(G)} {{|C(v)|} \choose 2} + \frac {1-\mu } 4 \sum_{
v \in V(G)} t(v) < \mu  \sum_{v \in V(G)} g(v).
$$
Thus
$$ \frac 14 \sum_{v \in V(G)}|C(v)|^2 + \frac {1-\mu } 4 \sum_{
v \in V(G)} t(v) < \mu  \sum_{v \in V(G)} g(v),$$
which implies that there exists some vertex $v$ such that
\begin{equation} \label{LowerBound}
\frac 14 |C(v)|^2 + \frac {1-\mu } 4 t(v) < \mu  g(v).
\end{equation}
By Lemmas~\ref{ModifiedLemma3.1}, we may also assume that
\begin{equation} \label{UpperBound}
g(v) < |C(v)|^2 (1+\mu) \left (
\frac {1+\mu + \sqrt{ (1+\mu)^2 + \frac { 4(1+\mu) e(v)} {|C(v)|^2}}} 2 + \frac {e(v)} {|C(v)|^2} \right )
\end{equation}
Combining (\ref{LowerBound}) with (\ref{UpperBound}),
$$\frac 14 |C(v)|^2 + \frac {1-\mu } 4 t(v) < |C(v)|^2 \mu (1+\mu) \left (
\frac {1+\mu + \sqrt{ (1+\mu)^2 + \frac { 4(1+\mu) e(v)} {|C(v)|^2}}} 2 + \frac {e(v)} {|C(v)|^2} \right ).$$
Since $e (v) \le t(v)$, we obtain
\begin{equation} \label{MainInequality}
\frac 14 < \mu (1+\mu)
\frac {1+\mu + \sqrt{ (1+\mu)^2 + \frac { 4(1+\mu) e(v)} {|C(v)|^2}}} 2+ \frac { 4  \mu ^2 + 5 \mu -1} 4 \cdot \frac {t(v)}
{|C(v)|^2}.
\end{equation}
The proof is now broken into two cases:

Case 1: $t(v) \ge |C(v)|^2/4$. Recall that $4\mu^2 + 5 \mu -1 \le 0$. Since $$e (v) \le |C_{A(v)}|\cdot |C_{B(v)}|
 = |C_{A(v)}|\cdot (|C(v)|- |C_{A(v)}|) \le  |C(v)|^2/4,$$
(\ref{MainInequality}) implies that
\begin{equation} \label{condition3}
\frac 14 < \mu (1+\mu)
\frac {1+\mu + \sqrt{ (1+\mu)^2 + 1+\mu } } 2+ \frac { 4  \mu ^2 + 5 \mu -1} {16}.
\end{equation}

Case 2: $t(v) \le |C(v)|^2/4$. Since $e (v) \le t(v)$,
(\ref{MainInequality}) implies that
$$\frac 14 < \mu (1+\mu)
\frac {1+\mu + \sqrt{ (1+\mu)^2 + \frac { 4(1+\mu) t(v)} {|C(v)|^2}}} 2+ \frac { 4  \mu ^2 + 5 \mu -1} 4 \cdot \frac {t(v)}
{|C(v)|^2} .$$
Define
$$f(x) = \mu (1+\mu)
\frac {1+\mu + \sqrt{ (1+\mu)^2 + 4(1+\mu) x }} 2+ \frac { (4  \mu ^2 + 5 \mu -1)x } 4,$$
where $0 \le x = t(v) /|C(v)|^2 \le 1/4$.
Taking the derivative of $f(x)$,
$$f^ \prime (x) = \frac {\mu (1+ \mu)^2}{\sqrt{ (1+\mu)^2 + 4(1+\mu) x }} + \frac { 4  \mu ^2 + 5 \mu -1} 4 \ge
 \frac {\mu (1+ \mu)^2}{\sqrt{ (1+\mu)^2 + 1+\mu  }} + \frac { 4  \mu ^2 + 5 \mu -1} 4.$$
It is easy to check that when $4  \mu ^2 + 5 \mu -1 \le 0$ we have
$$ \frac {\mu (1+ \mu)^2}{\sqrt{ (1+\mu)^2 + 1+\mu  }} + \frac { 4  \mu ^2 + 5 \mu -1} 4 \ge 0
\mbox{ iff } 24\mu^4 + 49 \mu^3 + 8 \mu^2 -19 \mu +2 \le 0.$$
Thus $f^\prime (x) \ge 0$, which implies that $f(x)$ is increasing. Thus
$$\frac 14 < f(x)  \le f\left (\frac 14 \right ) = \mu (1+\mu)
\frac {1+\mu + \sqrt{ (1+\mu)^2 + 1+\mu } } 2+ \frac { 4  \mu ^2 + 5 \mu -1} {16}.$$
By combining the above two cases, we always have (\ref{condition3}). Furthermore it is easy to check that,
when $8\mu^3+20\mu^2+13\mu-5 \le 0$,
(\ref{condition3}) is equivalent to $32\mu^4 - 8 \mu^3 -159 \mu^2 -130 \mu +25 <0, $ a contradiction. \hspace*{\fill}$\Box$
\end{proof}

\begin{thm} \label{Main}
If $G$ is a $3$-free digraph, then $\beta(G) < 0.8616 \gamma (G).$
\end{thm}

\begin{proof} We prove the theorem by induction on the number of vertices of $G$.
Set $\mu =0.16065$. Then $\mu$ satisfies all four inequalities in Theorem~\ref{condition}, and thus
there exists a vertex $v$ and a partition $V_1$, $V_2$, $\{v \}$ where the number of
missing edges, denoted $\rho$, is at least $(1+\mu)$ times the number of decycling edges, denoted $\tau$.
By induction hypothesis, $\beta (G[V_1]) < 0.8616 \gamma (G[V_1]) $ and $\beta (G[V_2]) < 0.8616 \gamma (G[V_2]) $.
Putting all these together yields
$$\hspace{.5cm} \beta(G) \le \gamma (G[V_1])  + \gamma (G[V_2]) + \tau < 0.8616 \gamma (G[V_1]) + 0.8616 \gamma (G[V_2])  + \rho / (1+\mu)
\le 0.8616 \gamma (G). \hspace{.5cm} \Box $$
\end{proof}

\vspace{.2cm}

{\bf Acknowledgement.} This is part of a research project done by three high school students (Chen, Karson, and Liu) in the summer of 2009
under the supervision of
Dr.~Jian Shen at Texas State University. Chen, Karson, and Liu thank Texas State Math Camp for providing this research opportunity. Shen wants to thank Professor Peter Hamburger for providing reference~\cite{DunkumHP}.

\bibliographystyle{amsplain}

\end{document}